\documentclass[11pt]{article}

\usepackage[margin=1in]{geometry}
\usepackage[T1]{fontenc}
\usepackage[utf8]{inputenc}
\usepackage{mathtools,amsthm}
\IfFormatAtLeastTF{2026-06-01}{}{\usepackage{aliascnt}}
\usepackage{lmodern,amssymb,mathrsfs}
\usepackage{microtype}
\usepackage{booktabs,array,tabularx}
\usepackage{enumitem}
\usepackage[backend=biber,style=authoryear,natbib=true,maxcitenames=2,maxbibnames=99,uniquelist=false]{biblatex}
\usepackage{xcolor}
\usepackage{hyperref}
\usepackage[nameinlink,capitalize,noabbrev]{cleveref}
\usepackage{titlesec}

\definecolor{linkblue}{RGB}{24,76,140}
\hypersetup{
  hidelinks,
  pdftitle={Sharp Minimax Rates for Highly Smooth Strongly Convex Zeroth-Order Optimization}
}
\setlist[itemize]{leftmargin=1.35em,itemsep=0.15em,topsep=0.25em}
\setlist[enumerate]{leftmargin=1.5em,itemsep=0.15em,topsep=0.25em}
\titleformat{\section}{\large\bfseries}{\thesection}{0.55em}{}
\titleformat{\subsection}{\normalsize\bfseries}{\thesubsection}{0.5em}{}
\allowdisplaybreaks[2]
\newtheorem{theorem}{Theorem}[section]
\IfFormatAtLeastTF{2026-06-01}{%
  \newcommand{\sharedtheorem}[2]{\newtheorem{#1}[theorem]{#2}}%
}{%
  \newcommand{\sharedtheorem}[2]{%
    \newaliascnt{#1}{theorem}%
    \newtheorem{#1}[#1]{#2}%
    \aliascntresetthe{#1}}%
}
\sharedtheorem{lemma}{Lemma}
\sharedtheorem{proposition}{Proposition}
\sharedtheorem{corollary}{Corollary}
\theoremstyle{definition}
\sharedtheorem{definition}{Definition}
\theoremstyle{remark}
\sharedtheorem{remark}{Remark}
\crefname{theorem}{Theorem}{Theorems}
\crefname{lemma}{Lemma}{Lemmas}
\crefname{proposition}{Proposition}{Propositions}
\crefname{corollary}{Corollary}{Corollaries}
\crefname{definition}{Definition}{Definitions}
\crefname{remark}{Remark}{Remarks}

\newcommand{\R}{\mathbb{R}}
\newcommand{\E}{\mathbb{E}}

\newcommand{\cA}{\mathcal{A}}
\newcommand{\cF}{\mathcal{F}}
\newcommand{\cE}{\mathcal{E}}
\newcommand{\cH}{\mathcal{H}}
\newcommand{\eps}{\varepsilon}
\newcommand{\op}{\mathrm{op}}
\newcommand{\KL}{\operatorname{KL}}
\newcommand{\TV}{\operatorname{TV}}
\newcommand{\Ham}{\operatorname{Ham}}

\newcommand{\Ball}{B_2^d}
\newcommand{\norm}[1]{\left\lVert #1\right\rVert}
\newcommand{\ip}[2]{\left\langle #1,#2\right\rangle}
\newcommand{\abs}[1]{\left|#1\right|}

\title{\textbf{Sharp Minimax Rates for Highly Smooth Strongly Convex Zeroth-Order Optimization}}
\newcommand{\authornames}{Haihan Zhang, Wendao Wu, Chenheng Zhang, Yanyi Li, Chunyuan Zheng, Cong Fang, Haoxuan Li, Zhouchen Lin}
\newcommand{\affiliation}{Peking University}
\newcommand{\contactemail}{colombo.2026@outlook.com}
\newcommand{\mail}[1]{\href{mailto:#1}{\texttt{#1}}}
\newif\ifanonymous
\anonymousfalse
\ifanonymous
\author{Anonymous Authors}
\hypersetup{pdfauthor={Anonymous Authors}}
\else
\author{Haihan Zhang$^{1,*}$\quad Wendao Wu$^{1,*}$\quad Chenheng Zhang$^{1,*}$\\[0.20em]
Yanyi Li$^1$\quad Chunyuan Zheng$^1$\\[0.20em]
Cong Fang$^{1}$\quad Haoxuan Li$^{1,\dagger}$\quad Zhouchen Lin$^{1,\dagger}$\\[0.60em]
\small $^1$\affiliation\\[0.35em]
\footnotesize\mail{\contactemail}\\[0.35em]
\small $^*$Equal contribution.\quad $^\dagger$Corresponding authors.}
\hypersetup{pdfauthor={\authornames}}
\fi
\date{}
\newcommand{\ResearchAgentSystem}{our laboratory's internal auto-research system}

\begin{document}
\maketitle

\begin{abstract}
We determine the minimax dimension dependence of noisy zeroth-order optimization for globally strongly convex functions with higher-order smoothness in multilinear operator norm. Each query returns one function value with fresh independent Gaussian noise. For dimension $d$, query budget $T$, and any fixed finite smoothness order $\beta>2$, the minimax expected objective error satisfies
\[
  \cE_\beta(T,d)
  =\Theta_\beta\!\left(\min\left\{1,
  \left(\frac{d^2}{T}\right)^{(\beta-1)/\beta}\right\}\right),
\]
with dimension-independent strong-convexity, gradient-smoothness, higher-order-smoothness, noise, and minimizer-radius bounds. Thus, for sufficiently small target error $\eps$, the necessary and sufficient number of noisy values is $\Theta_\beta(d^2\eps^{-\beta/(\beta-1)})$, with no logarithmic loss. The new lower bound matches the dimension and budget dependence of the upper bound of \citet{akhavan2024gradientfree}, and holds for arbitrary randomized adaptive algorithms with unrestricted query points. Its key step localizes a normalized hypercube with one radial cutoff: a query can emphasize one hidden coordinate, but the aggregate information about all coordinates remains small. A self-contained multiscale spherical estimator attains the matching upper bound by cancelling lower-order bias terms. Constants in the rates depend only on the fixed smoothness order under the stated normalization.
\end{abstract}


\begin{samepage}
\noindent\textbf{AI Usage.}
Nearly the entire research pipeline for this paper was carried out by
\ResearchAgentSystem{}, powered by GPT-5.6 Sol. The system also conducted a
Lean-backed article audit of the resulting manuscript. The authors subsequently reviewed and approved the
mathematical claims, presentation, and formal artifacts, and take
responsibility for the final manuscript. The complete Lean audit report and the system's technical report will be made public at a later date.
\par
\end{samepage}

\paragraph{Publication Intent and Recognition of Independent Work.}

This preprint is released as a research output of our study of
automated scientific research, for public scrutiny and discussion.
As we do not intend to submit this manuscript to a journal or
conference, it has not undergone the full editorial refinement
and internal review we would normally require for submission.
The authors remain responsible for its contents and for correcting
any errors or inaccurate attribution.

We recognize that releasing these results may affect researchers
who have independently been pursuing the same questions, including
work that is not yet public.
If this applies to your work, we invite you to contact us at
\texttt{colombo.2026@outlook.com}.
A brief description of the overlap, together with any dated
materials you are comfortable sharing, would help us understand
its scope and chronology.
Such materials might include draft excerpts, presentation slides,
or revision records; a complete manuscript or unpublished proof
is not required to begin the conversation.

We will review these cases carefully and update citations,
acknowledgments, and accounts of independent or concurrent work
as appropriate, including revising priority claims where warranted.
Nonpublic materials will be treated confidentially and used only
to assess the overlap unless their authors authorize otherwise.
We will obtain explicit permission from the authors before
publicly discussing such materials or providing them to an
automated research system.

\section{Introduction}\label{sec:intro}

We consider the optimization problem
\[
  \min_{x\in\R^d} f(x),
\]
where $f$ is strongly convex, has a Lipschitz gradient, and has higher-order smoothness of order $\beta>2$. Access to $f$ is limited to noisy values: a query at $z\in\R^d$ returns $f(z)+\xi$, with fresh independent noise $\xi\sim N(0,1)$. After $T$ adaptively chosen queries, the algorithm returns a point $\widehat x$. We seek the smallest worst-case expected objective error $\E[f(\widehat x)]-\min_x f(x)$ achievable by any randomized algorithm. The central issue is how this error depends on dimension $d$, budget $T$, and smoothness $\beta$. Higher-order regularity is measured in multilinear operator norm, with all remaining class parameters fixed independently of $d$ and $T$; \cref{sec:setup} gives the precise assumptions.

Optimization from function values is a classical problem in stochastic approximation and a basic model for optimization without derivative access \citep{kiefer1952stochastic,spall1992multivariate,larson2019derivativefree}. Its statistical difficulty comes from estimating useful directions from noisy scalar observations. Higher smoothness allows more accurate gradient approximations, but each observation must still convey information about a $d$-dimensional objective. This paper determines the dimension cost that remains unavoidable even when the objective has arbitrarily high, fixed finite smoothness.

The effect of smoothness on accuracy is well understood. Local differences become less biased as their radius shrinks, while independent observation noise is amplified. Higher-order cancellation improves this tradeoff and yields the characteristic error exponent $(\beta-1)/\beta$ for smoothness order $\beta>2$ \citep{polyak1990optimal,bach2016highly,akhavan2020exploiting}. Subsequent analyses sharpened the dependence on dimension \citep{novitskii2022improved,akhavan2024gradientfree}, but did not determine whether the best achievable dimension factor is necessary. Thus the remaining issue is whether better estimation can reduce the dimension cost, or whether the available lower bounds underestimate the information required by any algorithm.

To state this gap precisely, let $\cE_\beta(T,d)$ denote the minimax expected objective error in the model above, formally defined in \eqref{eq:risk}. All class parameters other than $d$ and $T$, including the finite order $\beta$, are fixed. For this smoothness model, the bounds of \citet{akhavan2024gradientfree}, expressed in individual function evaluations, give
\begin{equation}\label{eq:prior-bracket-intro}
  \Omega_\beta\!\left(dT^{-\frac{\beta-1}{\beta}}\right)
  \ \le\ \cE_\beta(T,d)\ \le\
  O_\beta\!\left(d^{2-\frac{2}{\beta}}T^{-\frac{\beta-1}{\beta}}\right)
\end{equation}
in the common high-budget regime $T\gtrsim_\beta d^\beta$. The dimension exponents agree at $\beta=2$, but leave a factor $d^{1-2/\beta}$ unresolved for every $\beta>2$.

We prove that the upper dimension dependence is necessary. For every fixed finite $\beta>2$ and all integers $d,T\ge1$, \cref{thm:main} establishes
\begin{equation}\label{eq:intro-main}
  \cE_\beta(T,d)
  =\Theta_\beta\!\left(\min\left\{1,
  d^{2-\frac{2}{\beta}}T^{-\frac{\beta-1}{\beta}}\right\}\right).
\end{equation}
The new lower bound holds for arbitrary randomized adaptive algorithms, with unrestricted query points and outputs, and matches the upper order established by \citet{akhavan2024gradientfree}. Equivalently, the minimax number $Q_\beta^*(d,\eps)$ of observations needed for sufficiently small expected objective error $\eps>0$ satisfies
\begin{equation}\label{eq:intro-query}
  Q_\beta^*(d,\eps)
  =\Theta_\beta\!\left(d^2\eps^{-\beta/(\beta-1)}\right).
\end{equation}
Higher smoothness therefore improves the accuracy exponent, while the dimension cost remains quadratic. The characterization also identifies the transition at $T\asymp d^2$: below this scale the worst-case error remains of constant order. \Cref{tab:comparison} situates the result within the progression of higher-order smoothing bounds and records smoothness assumptions that affect dimension-dependent comparisons.

The main obstacle is to enforce this information cost without restricting where an algorithm queries. A global linear perturbation of a quadratic is unsuitable: sufficiently distant observations can amplify its hidden coefficients. Moreover, separately maximizing the signal of each hidden coordinate can overestimate how much information a single query reveals. Testing-based optimization lower bounds provide a natural framework for converting limited information into error \citep{agarwal2012information,shamir2013complexity}; obtaining the sharp dimension factor here requires a hard family whose smoothness and aggregate signal can both be controlled independently of dimension.

Our construction places a linear perturbation with normalized direction $s/\sqrt d$, $s\in\{-1,+1\}^d$, inside one shared radial cutoff of radius $h$. Outside that ball, every hypothesis agrees with the same quadratic. Inside it, a query may emphasize one coordinate, but cannot simultaneously maximize the signal of all $d$ coordinates: the sum of squared differences from the one-bit neighboring hypotheses is at most $O_\beta(h^{2\beta}/d)$. Summing these signals before applying the adaptive Gaussian divergence calculation bounds the average information per bit by $O_\beta(Th^{2\beta}/d^2)$. Choosing $h\asymp_\beta\min\{1,(d^2/T)^{1/(2\beta)}\}$ leaves a constant fraction of the bits unresolved. Their displaced minimizers then force the error in \eqref{eq:intro-main}. The shared localization both prevents distant queries from escaping the information bound and preserves dimension-independent derivative bounds.

For completeness, we also provide a finite multiscale realization of the matching upper rate. It combines spherical smoothing with cancellation across a fixed collection of radii, building on the smoothing and higher-order estimation principles in \citet{flaxman2005online,nesterov2017random,bach2016highly}. This yields bias $O_\beta(h^{\beta-1})$ and second moment $O_\beta(d^2/h^2)$, which weighted projected stochastic gradient descent balances without a logarithmic loss. This construction supports the matching characterization; the new complexity conclusion comes from the lower bound. After discussing related work, we state the model and main result in \cref{sec:setup}, prove the new lower bound in \cref{sec:lower}, and discuss the fixed-order scope and endpoint in \cref{sec:scope}. The complete upper-bound construction and proof are deferred to \cref{sec:upper}.

\begin{table}[!t]
\centering
\setlength{\belowcaptionskip}{8pt}
\caption{Selected expected objective-error bounds for noisy strongly convex zeroth-order optimization with fixed $\beta>2$. Here $T$ counts individual noisy values, and other class parameters are fixed. All rows display leading terms in the common regime $T\gtrsim_\beta d^\beta$; the all-budget saturation of our rate is given in \cref{thm:main}. A dash means that no lower bound is compared from the selected result.}
\label{tab:comparison}
\small
\begingroup
\renewcommand{\arraystretch}{1.22}
\setlength{\tabcolsep}{4pt}
\begin{tabularx}{\textwidth}{@{}>{\raggedright\arraybackslash}p{0.215\textwidth}>{\centering\arraybackslash}p{0.28\textwidth}>{\centering\arraybackslash}p{0.18\textwidth}>{\raggedright\arraybackslash}X@{}}
\toprule
Work & Upper bound & Lower bound & Smoothness / scope \\
\midrule
\citet{bach2016highly}
& $O_\beta\!\bigl(d^{\frac{2(\beta-1)}{\beta+1}}T^{-\frac{\beta-1}{\beta+1}}\bigr)$
& ---
& Integer $\beta\ge3$; derivative smoothness; unconstrained \\
\addlinespace[5pt]
\citet{akhavan2020exploiting}
& $O_\beta(d^2T^{-\frac{\beta-1}{\beta}})$
& $\Omega_\beta(dT^{-\frac{\beta-1}{\beta}})$
& Taylor remainder; lower bound for a queried output \\
\addlinespace[5pt]
\citet{novitskii2022improved}
& $O_\beta(d^{2-\frac{1}{\beta}}T^{-\frac{\beta-1}{\beta}})$
& ---
& Taylor remainder; compact domain, locally Lipschitz objective \\
\addlinespace[5pt]
\citet{akhavan2024gradientfree}
& $O_\beta(d^{2-\frac{2}{\beta}}T^{-\frac{\beta-1}{\beta}})$
& $\Omega_\beta(dT^{-\frac{\beta-1}{\beta}})$
& Derivative H\"older in operator norm; arbitrary output \\
\midrule
\textbf{This Paper}\newline (\cref{thm:main})
& \multicolumn{2}{c}{$\boldsymbol{\Theta_\beta\!\left(d^{2-\frac{2}{\beta}}T^{-\frac{\beta-1}{\beta}}\right)}$}
& Same smoothness; new lower bound matches the upper rate \\
\bottomrule
\end{tabularx}
\endgroup
\end{table}

\section{Related Works}

\paragraph{Stochastic approximation and random perturbations.}
Optimization from noisy function values originates with \citet{kiefer1952stochastic}. Higher-order finite differences and refined asymptotic analyses appear in \citet{fabian1967stochastic}, while \citet{spall1992multivariate} develops simultaneous perturbation estimates using a small number of evaluations per iteration. \citet{dippon2003accelerated} studies randomized gradient estimates that exploit higher smoothness and analyzes their asymptotic distributions. These results establish the statistical value of smoothness and randomized directions; a dimension-sharp, nonasymptotic minimax characterization requires tracking how information is distributed across coordinates. For a broader account of derivative-free methods and their oracle assumptions, see \citet{larson2019derivativefree}.

\paragraph{Higher-order smoothness and dimension dependence.}
\citet{polyak1990optimal} establish the characteristic budget exponent $T^{-(\beta-1)/\beta}$. \citet{bach2016highly} develop kernel-weighted spherical differences for highly smooth optimization in stochastic and online settings. The dimension dependence of nonasymptotic strongly convex guarantees is subsequently improved by \citet{akhavan2020exploiting,novitskii2022improved,akhavan2024gradientfree}; \cref{tab:comparison} summarizes selected bounds. The single-node specialization of \citet{akhavan2021distributed} independently attains the $d^{2-1/\beta}$ dimension factor under scalar Taylor-remainder smoothness. That condition is also used by the 2020 and 2022 works, whereas \citet{akhavan2024gradientfree} analyze derivative-H\"older regularity in operator norm, including noninteger orders. Derivative-H\"older control implies a scalar Taylor-remainder bound with a constant depending only on $\beta$; the reverse comparison must not be presumed dimension-free \citep[Remark~2]{akhavan2024gradientfree}. Their sphere-based kernels and Poincar\'e analysis provide our direct upper benchmark. Our contribution is the matching lower dimension dependence for this operator-norm class; interpolation across finitely many radii supplies a self-contained upper construction. More recently, \citet{akhavan2026derivatives} study minimization of a partial derivative from noisy values of its primitive. Their direct-optimization specialization has upper dimension factor $d^{2-1/\beta}$, and their Remark~5.2 explicitly leaves sharp dimension dependence unresolved; it does not establish the operator-norm characterization proved here.

\paragraph{Smoothing and feedback models.}
Randomized smoothing connects value observations to gradients of a smoothed objective, both in bandit convex optimization \citep{flaxman2005online} and in gradient-free convex minimization \citep{nesterov2017random}. The statistical power of paired evaluations is studied by \citet{duchi2015optimal}, and \citet{shamir2017optimal} gives an optimal two-point method for convex Lipschitz objectives. Those guarantees depend on their feedback models: evaluating a shared random loss at two points differs from taking two scalar observations with fresh independent additive noise. In the latter model, differencing amplifies the noise as the radius shrinks, so two-point guarantees from a different oracle cannot be imported solely on the basis of the number of evaluations. Our complexity counts every independently corrupted scalar observation.

\paragraph{Minimax lower bounds and neighboring classes.}
Information-theoretic reductions connect stochastic optimization to statistical testing \citep{agarwal2012information}; hypercube testing is treated systematically by \citet{tsybakov2009introduction}. \citet{shamir2013complexity} already establishes a quadratic dimension cost in noisy value complexity for smooth strongly convex optimization. Our result shows that this cost persists under every fixed finite higher-order operator-norm smoothness condition considered here. The proof uses the same testing perspective, with a shared radial localization and an aggregate neighboring-signal estimate. A nearby matching result is due to \citet{yu2024stochastic}: for strongly convex functions with Hessians Lipschitz in Frobenius norm, their asymptotic minimax objective-error rate is $\Theta(dT^{-2/3})$ when the class parameters are fixed. A dimension-free operator-norm Hessian-Lipschitz constant does not imply a dimension-free Frobenius constant, so their result does not directly settle our $\beta=3$ class. Structural restrictions also change the problem: \citet{akhavan2025additive} establish the minimax rate $dT^{-(\beta-1)/\beta}$ for additive objectives under their assumptions, whereas our class imposes no such decomposition. These results identify optimal rates for neighboring classes; the new conclusion here concerns the dimension dependence for the general operator-norm class.

Local minimax analysis provides a complementary perspective: \citet{wang2018local} study noisy optimization in neighborhoods of reference functions on a compact domain. Their nearby objectives need not remain strongly convex even when the reference function is strongly convex; those local neighborhoods therefore differ from our uniformly strongly convex higher-order class.

\section{Problem formulation and main result}\label{sec:setup}

Fix a finite real number $\beta>2$ and define
\begin{equation}\label{eq:k-alpha}
  k:=\lceil\beta\rceil-1,
  \qquad
  \alpha:=\beta-k\in(0,1],
  \qquad
  p_\beta:=\frac{\beta-1}{\beta}.
\end{equation}
Thus $k$ is the largest integer strictly smaller than $\beta$; in particular, $\beta=4$ means $k=3$ and $\alpha=1$. We use the Euclidean norm and write $\Ball(r):=\{x\in\R^d:\norm{x}\le r\}$. For a $j$-linear form $A$, let
\[
  \norm{A}_{\op}
  :=\sup_{\norm{u_1},\ldots,\norm{u_j}\le1}
  \abs{A[u_1,\ldots,u_j]}.
\]

\begin{definition}[Normalized class]\label{def:class}
Let $\cF_\beta(d)$ be the class of functions $f:\R^d\to\R$ such that $f\in C^k$, the unique minimizer $x_f^\star$ belongs to $\Ball(1)$, and for all $x,y\in\R^d$,
\begin{align}
  f(y)&\ge f(x)+\ip{\nabla f(x)}{y-x}+\frac12\norm{y-x}^2,
  \label{eq:strong-convexity}\\
  \norm{\nabla f(x)-\nabla f(y)}&\le 3\norm{x-y},
  \label{eq:gradient-lipschitz}\\
  \norm{D^kf(x)-D^kf(y)}_{\op}&\le \norm{x-y}^{\alpha}.
  \label{eq:high-smoothness}
\end{align}
\end{definition}
The constants in \cref{def:class} are independent of $d,T$, and $\eps$.  They are chosen only for notational clarity; \cref{app:parameters} gives the fixed-parameter extension.  For $\beta=4$, \eqref{eq:high-smoothness} is operator-norm Lipschitzness of $D^3f$.  Since $\nabla f(x_f^\star)=0$, \eqref{eq:gradient-lipschitz} also implies the dimension-free local bound $\norm{\nabla f(z)}\le9$ on $\Ball(2)$.

At call $t$, an algorithm chooses $z_t\in\R^d$ adaptively and observes
\begin{equation}\label{eq:oracle}
  Y_t=f(z_t)+\xi_t,
  \qquad
  \xi_t\stackrel{\mathrm{iid}}{\sim}N(0,1).
\end{equation}
The algorithm may use arbitrary computation, memory, and independent randomization.  Query points and the final output are unrestricted.  Every scalar observation in \eqref{eq:oracle} counts as one raw query; in particular, evaluations at two points use independent noises.

Define the minimax optimization error
\begin{equation}\label{eq:risk}
  \cE_\beta(T,d)
  :=\inf_{\cA:\,\le T\text{ raw calls}}
  \sup_{f\in\cF_\beta(d)}
  \E_f\!\left[f(\widehat x)-f(x_f^\star)\right].
\end{equation}
The corresponding minimax raw query complexity is
$Q_\beta^*(d,\eps):=\inf\{T\in\mathbb N:\cE_\beta(T,d)\le\eps\}$.

\begin{theorem}[Sharp minimax rate]\label{thm:main}
For every fixed finite $\beta>2$, there exist constants $0<c_\beta\le C_\beta<\infty$, independent of $d$ and $T$, such that for all integers $d,T\ge1$,
\begin{equation}\label{eq:main-theorem}
  c_\beta
  \min\left\{1,\left(\frac{d^2}{T}\right)^{p_\beta}\right\}
  \le \cE_\beta(T,d)\le
  C_\beta
  \min\left\{1,\left(\frac{d^2}{T}\right)^{p_\beta}\right\}.
\end{equation}
The lower bound holds for every randomized, fully adaptive algorithm with unrestricted query points.  The upper bound uses
$2(1+\lfloor(k-1)/2\rfloor)$ raw values per stochastic-gradient iteration and queries only inside $\Ball(2)$.
\end{theorem}

\begin{corollary}[Raw query complexity]\label{cor:query}
There exists $\eps_\beta>0$ such that, for every $0<\eps\le\eps_\beta$,
\begin{equation}\label{eq:query-complexity}
  Q_\beta^*(d,\eps)
  =\Theta_\beta\!\left(d^2\eps^{-\beta/(\beta-1)}\right).
\end{equation}
In particular, the correct dimension exponent is $2$, and no logarithmic loss is necessary.
\end{corollary}

\begin{proof}
Write $p=p_\beta$, and let $c_\beta,C_\beta$ be the constants in \cref{thm:main}. For $0<\eps<c_\beta$, the inequality $\cE_\beta(T,d)\le\eps$ excludes the saturated branch and forces $T\ge d^2(c_\beta/\eps)^{1/p}$. Conversely, taking $T=\lceil d^2(C_\beta/\eps)^{1/p}\rceil$ makes the upper bound at most $\eps$. Restricting $\eps$ to a fixed sufficiently small interval absorbs the integer rounding into a constant depending only on $\beta$.
\end{proof}

\paragraph{The upper bound and the new lower bound.}
The upper rate in \cref{thm:main} is already attained by \citet[Corollary~21]{akhavan2024gradientfree}. Apply their result on $\Theta=\Ball(1)$, with diameter $2$, strong-convexity constant $1$, gradient-smoothness constant $3$, higher-order constant $1$, and noise variance $1$. Our global assumptions cover their perturbed query points, and $\norm{\nabla f(x)}\le6$ on $\Theta$. Independent Gaussian observations satisfy their noise assumptions. For $T\gtrsim_\beta d^2$, their budget condition holds and the additional $O_\beta(d/T)$ term is dominated by $(d^2/T)^{p_\beta}$. Counting each value in a pair as one query changes only a constant. For a self-contained verification of the all-budget statement, \cref{sec:upper} supplies a finite multiscale estimator, its projected stochastic-gradient analysis, and the small-budget branch that returns the origin. The main contribution is the matching lower bound proved next.

\section{Matching lower bound}\label{sec:lower}

We prove the lower half of \cref{thm:main} in three steps: construct a family with dimension-independent derivative bounds, control the information in every adaptive transcript, and convert uncertain signs into objective error. A compactly supported profile ensures that queries outside the hidden region reveal no additional information.  Fix a $C^\infty$ function $\chi:\R\to[0,1]$ satisfying
\begin{equation}\label{eq:cutoff-main}
  \chi(u)=1\quad(u\le1/4),
  \qquad
  \chi(u)=0\quad(u\ge1).
\end{equation}
For each unit vector $v\in\R^d$, define
\begin{equation}\label{eq:profile-main}
  F_v(w):=\ip{v}{w}\chi(\norm{w}^2).
\end{equation}

\begin{lemma}[Dimension-free radial profile]\label{lem:profile-main}
For every integer $j\ge0$, there exists $M_j<\infty$, depending only on $j$ and $\chi$, such that
\[
  \sup_{d\ge1}\sup_{\norm v=1}\sup_{w\in\R^d}
  \norm{D^jF_v(w)}_{\op}\le M_j.
\]
Moreover, there exists $H_\beta<\infty$, independent of $d$, such that
\begin{equation}\label{eq:profile-holder-main}
  \norm{D^kF_v(w)-D^kF_v(z)}_{\op}
  \le H_\beta\norm{w-z}^{\alpha}.
\end{equation}
\end{lemma}
The proof, given in \cref{app:profile}, expands derivatives of $\chi(\norm{w}^2)$ over set partitions with blocks of size one or two.  The number of terms depends on derivative order but not on dimension.  This is exactly where using multilinear operator norm, rather than Frobenius norm, matters.

Choose
\begin{equation}\label{eq:amplitude-main}
  0<a_\beta\le \min\left\{\frac14,M_2^{-1},H_\beta^{-1}\right\},
\end{equation}
with zero reciprocals omitted.  Write $a=a_\beta$.  For $s\in\{-1,+1\}^d$, set $v_s=s/\sqrt d$.  Given $h\in(0,1]$, define
\begin{equation}\label{eq:hard-main}
  f_s(x)
  :=\norm{x}^2-a h^\beta F_{v_s}(x/h)
  =\norm{x}^2-a h^{\beta-1}\ip{v_s}{x}
    \chi(\norm{x}^2/h^2).
\end{equation}

\begin{lemma}[Admissibility and exact minimizers]\label{lem:admissible-main}
Every $f_s$ in \eqref{eq:hard-main} belongs to $\cF_\beta(d)$.  Its unique global minimizer is
\begin{equation}\label{eq:minimizer-main}
  x_s^\star=\frac a2h^{\beta-1}v_s.
\end{equation}
Moreover, $f_s(x)=\norm{x}^2$ for every $\norm{x}\ge h$, independently of $s$.
\end{lemma}

\begin{proof}[Proof sketch]
The perturbation Hessian has norm at most $a h^{\beta-2}M_2\le1$, so
$I\preceq\nabla^2f_s\preceq3I$.  Its $k$th derivative increment is bounded by
\[
  a h^{\beta-k}H_\beta\norm{(x-y)/h}^{\alpha}
  =aH_\beta\norm{x-y}^{\alpha}
  \le\norm{x-y}^{\alpha}.
\]
Since $\norm{x_s^\star}=(a/2)h^{\beta-1}\le h/8$, the cutoff equals one in a neighborhood of $x_s^\star$, where
$\nabla f_s(x)=2x-ah^{\beta-1}v_s$.  Thus \eqref{eq:minimizer-main} is stationary and global strong convexity makes it the unique minimizer.  The support statement follows from \eqref{eq:cutoff-main}.  Full details appear in \cref{app:admissibility}.
\end{proof}

With class membership established, the remaining issue is how much one scalar query can reveal about the hidden signs. Bounding each neighbor separately would discard the common radial geometry. We instead aggregate all neighboring signals at the same query. Let $s^{(i)}$ be obtained from $s$ by flipping coordinate $i$.  Direct subtraction gives, for every $x\in\R^d$,
\begin{align}
  \sum_{i=1}^d\bigl(f_s(x)-f_{s^{(i)}}(x)\bigr)^2
  &=\frac{4a^2h^{2\beta-2}}{d}\norm{x}^2
    \chi(\norm{x}^2/h^2)^2 \notag\\
  &\le \frac{4a^2h^{2\beta}}{d}.
  \label{eq:summed-signal-main}
\end{align}
This uniform summed-signal inequality is the key step.  It prevents an adaptive learner from simultaneously extracting the maximal signal about every sign bit.

Fix any randomized adaptive algorithm and let $P_s$ denote the law of its full transcript, internal random seed, and output under $f_s$.  The conditional Gaussian KL identity gives
\begin{equation}\label{eq:adaptive-kl-main}
  \KL(P_s\Vert P_{s^{(i)}})
  =\frac12\E_s\sum_{t=1}^T
  \bigl(f_s(z_t)-f_{s^{(i)}}(z_t)\bigr)^2.
\end{equation}
Conditioned on the same history and seed, the action rule is common to both hypotheses; hence adaptivity introduces no additional likelihood-ratio term.  Summing \eqref{eq:adaptive-kl-main} over $i$, applying \eqref{eq:summed-signal-main} along the $P_s$ trajectory, and averaging over $s$ yields
\begin{equation}\label{eq:average-kl-main}
  \overline K
  :=\frac1{d2^d}\sum_s\sum_{i=1}^d
  \KL(P_s\Vert P_{s^{(i)}})
  \le \frac{2a^2Th^{2\beta}}{d^2}.
\end{equation}

Assouad's hypercube argument and Pinsker's inequality imply that every sign estimator $\widehat s$ satisfies
\begin{equation}\label{eq:assouad-main}
  \frac1{2^d}\sum_s\E_s\Ham(\widehat s,s)
  \ge \frac d2\left(1-\sqrt{\frac{\overline K}{2}}\right).
\end{equation}
Choose
\begin{equation}\label{eq:h-lower-main}
  h:=\min\left\{1,
  \left(\frac{d^2}{16a^2T}\right)^{1/(2\beta)}\right\}.
\end{equation}
Then $\overline K\le1/8$, so the average expected Hamming loss is at least $3d/8$.  Decode $\widehat s_i$ from the sign of $\widehat x_i$.  Each sign error incurs squared distance at least
$a^2h^{2\beta-2}/(4d)$ from \eqref{eq:minimizer-main}.  Strong convexity therefore gives
\begin{equation}\label{eq:lower-final-main}
  \frac1{2^d}\sum_s
  \E_s\!\left[f_s(\widehat x)-f_s(x_s^\star)\right]
  \ge \frac{3a^2}{64}h^{2\beta-2}
  \gtrsim_\beta
  \min\left\{1,\left(\frac{d^2}{T}\right)^{p_\beta}\right\}.
\end{equation}
This proves the lower half of \cref{thm:main}.  The complete transcript calculation and testing reduction are supplied in \cref{app:lower-full}.

\section{Scope, boundaries, and interpretation}\label{sec:scope}

\paragraph{The result is fixed-order and fixed-parameter.}
The notation $\Theta_\beta$ permits constants depending on the fixed finite order $\beta$.  More generally, the same exponents hold for fixed positive noise variance, fixed strong-convexity and smoothness constants, and a fixed minimizer-radius bound, provided the class is nondegenerate; see \cref{app:parameters}.  We do not claim a jointly sharp formula in all of these parameters, nor a theorem uniform as $\beta\to\infty$ or the noise variance tends to zero.

\paragraph{$\beta=2$ is the natural boundary.}
The rate specializes to $\Theta(dT^{-1/2})$ at the level of exponents, agreeing with the usual globally smooth strongly convex problem \citep{shamir2013complexity,akhavan2024gradientfree}. The endpoint class must, however, be defined separately: keep strong convexity $1$ and gradient Lipschitz constant $3$, and omit \eqref{eq:high-smoothness}. Substituting $\beta=2$ directly into that extra condition would instead impose gradient Lipschitz constant $1$ and collapse the class to quadratics with Hessian $I$, for which unrestricted distant queries change the problem. For the nondegenerate endpoint class, the lower construction remains valid with perturbation amplitude $ah^2$ and sufficiently small $a$, while the single-scale estimator has bias at most $3h$; the same information and SGD arguments give $\Theta(\min\{1,d/\sqrt T\})$. For $1<\beta<2$, the naive global extension is empty.

\begin{proposition}[No globally strongly convex sub-Lipschitz-gradient class]\label{prop:beta-below-two}
Let $0<\nu<1$.  There is no differentiable function $f:\R^d\to\R$ that is globally $\mu$-strongly convex for some $\mu>0$ and whose gradient satisfies
$\norm{\nabla f(x)-\nabla f(y)}\le L\norm{x-y}^{\nu}$ for a finite $L$ and all $x,y\in\R^d$.
\end{proposition}

\begin{proof}
Strong convexity and Cauchy--Schwarz imply
$\norm{\nabla f(x)-\nabla f(y)}\ge\mu\norm{x-y}$.  Combining the two inequalities gives
$\mu r\le Lr^\nu$ for every distance $r>0$, or $\mu r^{1-\nu}\le L$, which fails as $r\to\infty$.
\end{proof}
Thus a meaningful $1<\beta<2$ theory requires changing the model, for example by restricting the domain to a bounded convex set.  It is not an omitted case of \cref{thm:main}.

\paragraph{Why the lower bound is not contradicted by far-away queries.}
All hypotheses in \eqref{eq:hard-main} coincide with the same quadratic outside $\Ball(h)$.  Querying farther away therefore reveals no information about $s$.  Inside the support, \eqref{eq:summed-signal-main} is uniform over the query location.  The lower bound consequently covers unrestricted queries without imposing a span, precision, memory, or computational restriction.

\paragraph{Why neighboring results do not close the same class.}
The Frobenius-norm Lipschitz-Hessian condition in \citet{yu2024stochastic} is geometrically stronger than a dimension-free operator-norm condition.  Exact values, common-noise two-point differences, derivative observations, and additive objectives also define different statistical experiments.  The theorem here concerns fresh independent noisy scalar values and expected final optimization error, not cumulative bandit regret.

\section{Conclusion}

For noisy zeroth-order optimization of globally strongly convex functions with fixed finite derivative-H\"older smoothness in multilinear operator norm, the minimax dimension exponent is exactly two in raw query complexity.  The polynomial gap in the direct prior benchmark is closed by a stronger lower bound, while a finite multiscale estimator gives a matching self-contained upper bound.  The decisive lower-bound principle is to localize a normalized hypercube with one shared radial cutoff and to aggregate neighboring signals before applying the adaptive information inequality.  This mechanism may be useful in other oracle models where a learner can concentrate a query in one direction but cannot maximize information about all hidden coordinates simultaneously.

\printbibliography

\clearpage
\appendix

\section{Dimension-free radial profile}\label{app:profile}

For completeness, one explicit cutoff satisfying \eqref{eq:cutoff-main} is obtained from
\[
  \eta(t)=
  \begin{cases}
    \exp(-1/t),&t>0,\\
    0,&t\le0,
  \end{cases}
  \qquad
  \chi(u)=\frac{\eta(1-u)}{\eta(1-u)+\eta(u-1/4)}.
\]
The denominator is positive for every $u$, and $\chi\in C^\infty$ with bounded derivatives.

\begin{proof}[Proof of \cref{lem:profile-main}]
Let $q(w)=\norm{w}^2$.  Its only nonzero derivatives are
\[
  Dq(w)[u]=2\ip{w}{u},
  \qquad
  D^2q(w)[u,u']=2\ip{u}{u'}.
\]
On the support of $\chi(q(w))$, $\norm{w}\le1$, so both multilinear operator norms are at most $2$.  Repeated differentiation gives
\begin{equation}\label{eq:partition-expansion-app}
  D^j(\chi\circ q)(w)[u_1,\ldots,u_j]
  =\sum_{\pi\in\mathscr P_j}
  \chi^{(|\pi|)}(q(w))
  \prod_{B\in\pi}D^{|B|}q(w)[u_i:i\in B],
\end{equation}
where $\mathscr P_j$ is the finite set of partitions of $\{1,\ldots,j\}$ into blocks of size one or two.  Formula \eqref{eq:partition-expansion-app} follows inductively: differentiating the outer factor creates a singleton block, while differentiating a first derivative of $q$ creates a pair; all higher derivatives of $q$ vanish.

Set $K_0=\|\chi\|_\infty=1$. For $j\ge1$, unit directions, and $\norm{w}\le1$,
\[
  \abs{D^j(\chi\circ q)(w)[u_1,\ldots,u_j]}
  \le K_j
  :=\sum_{\pi\in\mathscr P_j}
  2^{|\pi|}\norm{\chi^{(|\pi|)}}_\infty.
\]
Because $w\mapsto\ip{v}{w}$ is linear, the product rule yields, for $j\ge1$,
\[
  \norm{D^jF_v(w)}_{\op}\le K_j+jK_{j-1},
\]
uniformly over dimension and unit $v$.  Outside the unit ball, $F_v$ and all derivatives vanish.  Hence the first claim holds with $M_0=1$ and $M_j=K_j+jK_{j-1}$.

The fundamental theorem of calculus and the derivative bounds give
\[
  \norm{D^kF_v(w)-D^kF_v(z)}_{\op}
  \le\min\{2M_k,M_{k+1}\norm{w-z}\}.
\]
If $\norm{w-z}\le1$, use $\norm{w-z}\le\norm{w-z}^{\alpha}$; if it is larger than one, use $1\le\norm{w-z}^{\alpha}$.  Therefore \eqref{eq:profile-holder-main} holds with
$H_\beta=\max\{2M_k,M_{k+1}\}$.
\end{proof}

For $\beta=4$, writing $A_j=\norm{\chi^{(j)}}_\infty$, direct expansion gives the useful explicit estimates
\begin{equation}\label{eq:explicit-M-app}
  M_2\le6A_1+4A_2,
  \qquad
  M_4\le60A_2+80A_3+16A_4.
\end{equation}
Thus $D^3F_v$ is $M_4$-Lipschitz in multilinear operator norm, uniformly in $d$.

\section{Admissibility of the hard family}\label{app:admissibility}

\begin{proof}[Full proof of \cref{lem:admissible-main}]
The perturbation in \eqref{eq:hard-main} is $-ah^\beta F_{v_s}(x/h)$.  Its Hessian has operator norm at most
\[
  ah^{\beta-2}M_2\le aM_2\le1,
\]
because $h\le1$ and $\beta>2$.  Since the base quadratic has Hessian $2I$,
\[
  I\preceq\nabla^2f_s(x)\preceq3I
  \qquad\text{for all }x\in\R^d.
\]
This proves \eqref{eq:strong-convexity} and \eqref{eq:gradient-lipschitz}.

The $k$th derivative of the perturbation scales as $ah^{\beta-k}$.  The $k$th derivative increment of the base quadratic is zero: if $k=2$ it is constant, and if $k>2$ it vanishes.  Therefore, by \eqref{eq:profile-holder-main},
\begin{align*}
  \norm{D^kf_s(x)-D^kf_s(y)}_{\op}
  &\le ah^{\beta-k}H_\beta
  \norm{(x-y)/h}^{\alpha}\\
  &=aH_\beta h^{\beta-k-\alpha}
  \norm{x-y}^{\alpha}
  \le\norm{x-y}^{\alpha},
\end{align*}
since $\alpha=\beta-k$ and $aH_\beta\le1$.

Now
\[
  \norm{x_s^\star}
  =\frac a2h^{\beta-1}
  \le\frac18h^{\beta-1}
  \le\frac h8<\frac h2.
\]
Hence $\norm{x_s^\star}^2/h^2<1/4$, so the cutoff is identically one in a neighborhood of $x_s^\star$.  In this neighborhood,
\[
  \nabla f_s(x)=2x-ah^{\beta-1}v_s,
\]
and \eqref{eq:minimizer-main} is stationary.  Global strong convexity makes it the unique global minimizer.  It lies in $\Ball(1)$.  Finally, if $\norm{x}\ge h$, then $\chi(\norm{x}^2/h^2)=0$, and $f_s(x)=\norm{x}^2$.
\end{proof}

\section{Full adaptive lower-bound proof}\label{app:lower-full}

\subsection{Gaussian transcript divergence}

Fix an arbitrary randomized adaptive algorithm.  Its internal randomness may be sampled at time zero as a seed $R$ whose law is common to all hypotheses.  Let
\[
  H_{t-1}=(R,z_1,Y_1,\ldots,z_{t-1},Y_{t-1})
\]
be the history before call $t$, and let $z_t$ be any measurable function of $H_{t-1}$. Let $P_s$ denote the joint law of $R$, the full transcript, and the output under $f_s$; any randomness used to form the output is included in $R$.

\begin{lemma}[Adaptive Gaussian KL identity]\label{lem:adaptive-kl-app}
For every neighboring pair $(s,s^{(i)})$,
\[
  \KL(P_s\Vert P_{s^{(i)}})
  =\frac12\E_s\sum_{t=1}^T
  \bigl(f_s(z_t)-f_{s^{(i)}}(z_t)\bigr)^2.
\]
\end{lemma}

\begin{proof}
Conditioned on a common realized history, the action-selection kernel is identical under both hypotheses.  The conditional observation laws are
$N(f_s(z_t),1)$ and $N(f_{s^{(i)}}(z_t),1)$, whose KL divergence is half the squared difference of means.  Applying the chain rule for relative entropy and then taking expectation under $P_s$ gives the result.  Appending the algorithm's output through a common conditional kernel contributes zero additional KL.  Algorithms using fewer than $T$ calls may be padded by ignored calls at zero, where every $f_s$ has the same value.
\end{proof}

For a neighboring pair,
\[
  f_s(x)-f_{s^{(i)}}(x)
  =-\frac{2ah^{\beta-1}}{\sqrt d}s_ix_i
  \chi(\norm{x}^2/h^2).
\]
Therefore \eqref{eq:summed-signal-main} holds.  Applying it at every adaptive query along the $P_s$ trajectory gives
\begin{align*}
  \sum_{i=1}^d\KL(P_s\Vert P_{s^{(i)}})
  &=\frac12\E_s\sum_{t=1}^T
  \sum_{i=1}^d
  \bigl(f_s(z_t)-f_{s^{(i)}}(z_t)\bigr)^2\\
  &\le \frac{2a^2Th^{2\beta}}{d}.
\end{align*}
Averaging over $s$ and dividing by $d$ yields \eqref{eq:average-kl-main}.

\subsection{Hypercube testing}

\begin{lemma}[Assouad reduction]\label{lem:assouad-app}
For any estimator $\widehat s\in\{-1,+1\}^d$ based on the transcript,
\[
  \frac1{2^d}\sum_s\E_s\Ham(\widehat s,s)
  \ge\frac d2\left(1-\sqrt{\frac{\overline K}{2}}\right).
\]
\end{lemma}

\begin{proof}
Fix coordinate $i$ and pair vertices along the $i$th hypercube edges.  If $A$ is the event that the decision for bit $i$ equals $s_i$, then the sum of the two edgewise error probabilities is
\[
  1-P_s(A)+P_{s^{(i)}}(A)
  \ge1-\TV(P_s,P_{s^{(i)}}).
\]
Summing over undirected edges and coordinates gives
\[
  \frac1{2^d}\sum_s\E_s\Ham(\widehat s,s)
  \ge\frac d2\left(
  1-\frac1{d2^d}\sum_s\sum_i
  \TV(P_s,P_{s^{(i)}})
  \right).
\]
Pinsker's inequality
$\TV(P,Q)\le\sqrt{\KL(P\Vert Q)/2}$ and concavity of the square root then imply the claim.  Pinsker's inequality may, for example, be derived by applying the log-sum inequality to the event attaining total variation; see \citet{tsybakov2009introduction}.
\end{proof}

With $h$ as in \eqref{eq:h-lower-main}, \eqref{eq:average-kl-main} gives $\overline K\le1/8$, and therefore
\begin{equation}\label{eq:hamming-app}
  \frac1{2^d}\sum_s\E_s\Ham(\widehat s,s)
  \ge\frac{3d}{8}.
\end{equation}
Given any optimization output $\widehat x$, define
$\widehat s_i=+1$ when $\widehat x_i\ge0$ and $-1$ otherwise.  Since
\[
  (x_s^\star)_i
  =\frac{ah^{\beta-1}}{2\sqrt d}s_i,
\]
each sign error contributes at least $a^2h^{2\beta-2}/(4d)$ to squared Euclidean error.  By $1$-strong convexity,
\begin{align*}
  \frac1{2^d}\sum_s
  \E_s[f_s(\widehat x)-f_s(x_s^\star)]
  &\ge\frac1{2^{d+1}}\sum_s
  \E_s\norm{\widehat x-x_s^\star}^2\\
  &\ge\frac{a^2h^{2\beta-2}}{8d}
  \frac1{2^d}\sum_s\E_s\Ham(\widehat s,s)\\
  &\ge\frac{3a^2}{64}h^{2\beta-2}.
\end{align*}
If the minimum in \eqref{eq:h-lower-main} selects the second term, then
\[
  h^{2\beta-2}
  =\left(\frac{d^2}{16a^2T}\right)^{p_\beta}
  \ge\left(\frac{d^2}{T}\right)^{p_\beta}
\]
because $a\le1/4$.  If it selects one, the same quantity is at least the saturated minimum.  The finite family is fixed before the algorithm is run and is independent of the realized seed and queries.  Taking the supremum over $\cF_\beta(d)$ and then the infimum over algorithms proves the lower half of \cref{thm:main}.

\section{Matching upper bound}\label{sec:upper}

We now prove the upper half of \cref{thm:main}. The rate agrees with the benchmark of \citet{akhavan2024gradientfree}; the purpose here is a self-contained finite multiscale construction. A single central difference retains even-order smoothing bias. Combining finitely many radii cancels that bias, after which weighted projected stochastic gradient descent controls the residual error.

\subsection{A multiscale spherical estimator}

Let $U$ be uniform on the sphere $S^{d-1}$ and $V$ uniform on the unit ball $\Ball(1)$.  For $r>0$, define
\begin{equation}\label{eq:single-scale-main}
  G_r(x;U)
  :=\frac{d}{2r}
  \bigl(Y(x+rU)-Y(x-rU)\bigr)U,
\end{equation}
using two fresh noises.  By symmetry and the divergence theorem,
\begin{equation}\label{eq:ball-mean-main}
  m_r(x):=\E G_r(x;U)=\E_V\nabla f(x+rV).
\end{equation}

Set
\begin{equation}\label{eq:J-q-main}
  J:=\left\lfloor\frac{k-1}{2}\right\rfloor,
  \qquad q:=J+1,
  \qquad r_j:=2^{-j}\quad(0\le j\le J).
\end{equation}
A vector Taylor expansion of \eqref{eq:ball-mean-main} has the form
\begin{equation}\label{eq:mean-expansion-main}
  m_r(x)=\nabla f(x)+\sum_{\ell=1}^Jr^{2\ell}b_\ell(x)+e_r(x),
  \qquad
  \norm{e_r(x)}\le C_{\mathrm{tay},\beta}r^{\beta-1},
\end{equation}
where only even powers survive ball averaging and
\begin{equation}\label{eq:taylor-constant-main}
  C_{\mathrm{tay},\beta}
  =\frac{\Gamma(\alpha+1)}{\Gamma(\beta)}.
\end{equation}
Choose Lagrange weights
\begin{equation}\label{eq:weights-main}
  w_j:=\prod_{\substack{0\le\ell\le J\\\ell\ne j}}
  \frac{-r_\ell^2}{r_j^2-r_\ell^2},
\end{equation}
so that
\begin{equation}\label{eq:moment-cancellation-main}
  \sum_{j=0}^Jw_j=1,
  \qquad
  \sum_{j=0}^Jw_jr_j^{2\ell}=0
  \quad(1\le\ell\le J).
\end{equation}
At bandwidth $h\in(0,1]$, use one common direction $U$ and independent noises at all queried points, and set
\begin{equation}\label{eq:multiscale-main}
  \widehat g_h(x):=\sum_{j=0}^Jw_jG_{r_jh}(x;U).
\end{equation}
This estimator costs exactly $2q$ raw values.

\begin{proposition}[Bias and second moment]\label{prop:estimator-main}
For every $x\in\Ball(1)$ and $h\in(0,1]$,
\begin{align}
  \norm{\E\widehat g_h(x)-\nabla f(x)}
  &\le B_\beta h^{\beta-1},
  \label{eq:bias-main}\\
  \E\norm{\widehat g_h(x)}^2
  &\le V_\beta\frac{d^2}{h^2},
  \label{eq:moment-main}
\end{align}
where $B_\beta,V_\beta<\infty$ depend only on the fixed order $\beta$.
\end{proposition}

\begin{proof}[Proof sketch]
Equations \eqref{eq:mean-expansion-main} and \eqref{eq:moment-cancellation-main} cancel every explicit even Taylor term, leaving
\[
  B_\beta
  =C_{\mathrm{tay},\beta}
  \sum_{j=0}^J\abs{w_j}r_j^{\beta-1}.
\]
For the second moment, every query lies in $\Ball(2)$, where $\norm{\nabla f}\le9$.  The noise-free part of each $G_{r_jh}$ has norm at most $9d$.  The independent Gaussian difference at scale $j$ contributes variance $d^2/(2r_j^2h^2)$.  Hence one may take
\[
  V_\beta
  =81\left(\sum_j\abs{w_j}\right)^2
  +\frac12\sum_j\frac{w_j^2}{r_j^2}.
\]
The full Taylor and smoothing derivation is in \cref{app:upper-estimator}.
\end{proof}

\subsection{Weighted projected stochastic gradient descent}

Let $n$ be the number of stochastic-gradient iterations, and let $\cH_t$ contain all randomness and observations before iteration $t$. At each iteration draw a fresh direction $U_t$, independently of $\cH_t$, and use fresh noises at all $2q$ points. Starting from $x_1=0$, update
\begin{equation}\label{eq:algorithm-update-main}
  x_{t+1}
  =\Pi_{\Ball(1)}\!\left(
  x_t-\frac4{t+1}\widehat g_h(x_t)
  \right),
  \qquad 1\le t\le n,
\end{equation}
and return
\begin{equation}\label{eq:algorithm-output-main}
  \overline x_n
  :=\frac2{n(n+1)}\sum_{t=1}^ntx_t.
\end{equation}

\begin{lemma}[Biased stochastic-gradient bound]\label{lem:biased-sgd-main}
Suppose $f$ is $1$-strongly convex on $\Ball(1)$, its minimizer $x^\star$ belongs to $\Ball(1)$, and the conditional gradient estimates satisfy
\[
  \norm{\E[g_t\mid\cH_t]-\nabla f(x_t)}\le B,
  \qquad
  \E[\norm{g_t}^2\mid\cH_t]\le M.
\]
Then the iterates \eqref{eq:algorithm-update-main}, with $\widehat g_h(x_t)$ replaced by $g_t$, and output \eqref{eq:algorithm-output-main} satisfy
\begin{equation}\label{eq:biased-sgd-main}
  \E\!\left[f(\overline x_n)-f(x^\star)\right]
  \le B^2+\frac{4M}{n+1}.
\end{equation}
\end{lemma}
The bias is absorbed into the strong-convexity term, and weights proportional to the iteration index telescope the remaining squared distances; see \cref{app:sgd}.

Applying \cref{prop:estimator-main,lem:biased-sgd-main} gives, for every fixed $h\in(0,1]$,
\begin{equation}\label{eq:upper-tradeoff-main}
  \sup_{f\in\cF_\beta(d)}
  \E\!\left[f(\overline x_n)-f(x_f^\star)\right]
  \le B_\beta^2h^{2\beta-2}
  +\frac{4V_\beta d^2}{(n+1)h^2}.
\end{equation}
For $n\ge d^2$, choose
\begin{equation}\label{eq:h-upper-main}
  h=\left(\frac{d^2}{n}\right)^{1/(2\beta)}.
\end{equation}
The two terms in \eqref{eq:upper-tradeoff-main} then have the same order $(d^2/n)^{p_\beta}$. For $T\ge2qd^2$, take $n=\lfloor T/(2q)\rfloor\ge d^2$; then $n\ge T/(4q)$ and at most $T$ raw values are used. For $T<2qd^2$, return $0$, whose loss is at most $3/2$ because $x_f^\star\in\Ball(1)$ and the gradient is $3$-Lipschitz. In this range, $\min\{1,(d^2/T)^{p_\beta}\}\ge(2q)^{-p_\beta}$. Thus the same rate holds after enlarging a constant depending only on $\beta$, including budgets too small for one full iteration. This proves the upper half of \cref{thm:main} for all $T$.

\subsection{The explicit estimator for \texorpdfstring{$\beta=4$}{beta = 4}}

When $\beta=4$, we have $J=1$, $(r_0,r_1)=(1,1/2)$, and $(w_0,w_1)=(-1/3,4/3)$.  Four fresh values give
\begin{equation}\label{eq:beta4-main}
  \widehat g_h(x)
  =\frac{d}{6h}
  \left[
  8\bigl(Y(x+\tfrac h2U)-Y(x-\tfrac h2U)\bigr)
  -\bigl(Y(x+hU)-Y(x-hU)\bigr)
  \right]U.
\end{equation}
The $r^2$ term in the ball-smoothed gradient cancels exactly.  Under the normalization in \cref{def:class},
\begin{equation}\label{eq:beta4-bounds-main}
  \norm{\E\widehat g_h(x)-\nabla f(x)}\le\frac{h^3}{12},
  \qquad
  \E\norm{\widehat g_h(x)}^2\le229\frac{d^2}{h^2}.
\end{equation}
Consequently,
\[
  \E[f(\overline x_n)-f(x^\star)]
  \le\frac{h^6}{144}+\frac{916d^2}{(n+1)h^2},
\]
and, for $n\ge d^2$, $h=(d^2/n)^{1/8}$ yields the stated $d^{3/2}n^{-3/4}$ rate.

\subsection{Estimator proof details}\label{app:upper-estimator}

\subsubsection{Ball-smoothing identity}

\begin{lemma}\label{lem:ball-smoothing-app}
For $G_r$ defined in \eqref{eq:single-scale-main},
$\E G_r(x;U)=\E_V\nabla f(x+rV)$.
\end{lemma}

\begin{proof}
The observation noises have zero mean.  Symmetry under $U\mapsto-U$ gives
\[
  \E G_r(x;U)=\frac dr\E_U[f(x+rU)U].
\]
Apply the divergence theorem componentwise to $v\mapsto f(x+rv)$ over the unit ball.  Since the surface-area-to-volume ratio of the Euclidean unit ball is $d$, the right-hand side equals $\E_V\nabla f(x+rV)$.  For $d=1$, the same identity is the fundamental theorem of calculus on $[-1,1]$.
\end{proof}

\subsubsection{Vector Taylor expansion}

Let $m=k-1\ge1$.  Applying the integral Taylor formula to the vector field $g=\nabla f$ gives
\begin{equation}\label{eq:vector-taylor-app}
  \nabla f(x+z)
  =\sum_{j=0}^{m}\frac1{j!}D^j\nabla f(x)[z^j]+R_x(z),
\end{equation}
with
\begin{equation}\label{eq:taylor-remainder-app}
  R_x(z)
  =\frac1{(m-1)!}\int_0^1(1-t)^{m-1}
  \left(D^m\nabla f(x+tz)-D^m\nabla f(x)\right)[z^m]dt.
\end{equation}
Here $D^m\nabla f$ is vector-valued, whereas $D^kf$ is a scalar $k$-linear form. Their identification uses Euclidean duality:
\[
  \ip{u}{D^m\nabla f(x)[z_1,\ldots,z_m]}
  =D^kf(x)[z_1,\ldots,z_m,u].
\]
Taking the supremum over $\norm u\le1$ shows that \eqref{eq:high-smoothness} controls the vector remainder without an additional dimension factor. Consequently,
\begin{align*}
  \norm{R_x(z)}
  &\le\frac{\norm z^{m+\alpha}}{(m-1)!}
  \int_0^1(1-t)^{m-1}t^\alpha dt\\
  &=\frac{\Gamma(\alpha+1)}{\Gamma(m+\alpha+1)}
  \norm z^{\beta-1}
  =C_{\mathrm{tay},\beta}\norm z^{\beta-1}.
\end{align*}
Averaging \eqref{eq:vector-taylor-app} at $z=rV$ eliminates every odd-degree term.  The surviving even powers are precisely $r^{2\ell}$ for $1\le\ell\le J$, and the averaged remainder has norm at most $C_{\mathrm{tay},\beta}r^{\beta-1}$.  This proves \eqref{eq:mean-expansion-main}.

\subsubsection{Cancellation weights and second moment}

The values $w_j$ in \eqref{eq:weights-main} are the Lagrange interpolation coefficients evaluating at zero a polynomial specified at nodes $r_j^2$.  Hence \eqref{eq:moment-cancellation-main} holds.  Combining it with \eqref{eq:mean-expansion-main} proves \eqref{eq:bias-main} with
\begin{equation}\label{eq:Bbeta-app}
  B_\beta
  :=C_{\mathrm{tay},\beta}
  \sum_{j=0}^J\abs{w_j}r_j^{\beta-1}.
\end{equation}

For $x\in\Ball(1)$ and $h\le1$, every line segment between $x-r_jhU$ and $x+r_jhU$ lies in $\Ball(2)$.  Because $x_f^\star\in\Ball(1)$ and the gradient is $3$-Lipschitz,
\[
  \norm{\nabla f(z)}
  \le3\norm{z-x_f^\star}\le9
  \qquad(z\in\Ball(2)).
\]
The mean-value theorem therefore bounds the noise-free part of each $G_{r_jh}$ by $9d$ in norm.  The noise component at scale $j$ is
\[
  \frac{d}{2r_jh}(\xi_{j,+}-\xi_{j,-})U,
\]
whose second moment is $d^2/(2r_j^2h^2)$.  The noises are independent across scales and have zero mean conditional on $U$, so all noise cross terms and signal--noise cross terms vanish.  Thus
\begin{align*}
  \E\norm{\widehat g_h(x)}^2
  &\le81d^2\left(\sum_j\abs{w_j}\right)^2
  +\frac{d^2}{2h^2}\sum_j\frac{w_j^2}{r_j^2}\\
  &\le V_\beta\frac{d^2}{h^2},
\end{align*}
where
\begin{equation}\label{eq:Vbeta-app}
  V_\beta
  :=81\left(\sum_j\abs{w_j}\right)^2
  +\frac12\sum_j\frac{w_j^2}{r_j^2}.
\end{equation}
This proves \cref{prop:estimator-main}.

\subsection{Biased stochastic-gradient telescope}\label{app:sgd}

\begin{proof}[Proof of \cref{lem:biased-sgd-main}]
Let
$R_t=\E\norm{x_t-x^\star}^2$ and
$\Delta_t=\E[f(x_t)-f(x^\star)]$.
Projection nonexpansiveness gives
\[
  \E[\norm{x_{t+1}-x^\star}^2\mid\cH_t]
  \le\norm{x_t-x^\star}^2
  -2\eta_t\ip{\E[g_t\mid\cH_t]}{x_t-x^\star}
  +\eta_t^2M,
\]
where $\eta_t=4/(t+1)$.  Strong convexity implies
\[
  \ip{\nabla f(x_t)}{x_t-x^\star}
  \ge f(x_t)-f(x^\star)+\frac12\norm{x_t-x^\star}^2.
\]
If $b_t=\E[g_t\mid\cH_t]-\nabla f(x_t)$, then
\[
  2\abs{\ip{b_t}{x_t-x^\star}}
  \le2B\norm{x_t-x^\star}
  \le\frac12\norm{x_t-x^\star}^2+2B^2.
\]
Taking expectations yields
\[
  R_{t+1}
  \le(1-\eta_t/2)R_t-2\eta_t\Delta_t
  +2\eta_tB^2+\eta_t^2M.
\]
Rearranging at $\eta_t=4/(t+1)$ gives
\[
  \Delta_t
  \le\frac{(t-1)R_t-(t+1)R_{t+1}}8
  +B^2+\frac{2M}{t+1}.
\]
Multiply by $t$ and sum from $1$ to $n$.  The distance terms telescope to
$-n(n+1)R_{n+1}/8\le0$, while
$\sum_{t=1}^n t/(t+1)\le n$.  Therefore
\[
  \sum_{t=1}^nt\Delta_t
  \le\frac{n(n+1)}2B^2+2Mn.
\]
Convexity of $f$ at the weighted average \eqref{eq:algorithm-output-main} proves
\[
  \E[f(\overline x_n)-f(x^\star)]
  \le B^2+\frac{4M}{n+1}.
\]
No small-step assumption is needed because projection and the uniform second-moment bound control the initial steps.
\end{proof}

\subsection{Explicit raw-call audit for \texorpdfstring{$\beta=4$}{beta = 4}}\label{app:beta4}

For $\beta=4$, the ball-smoothed gradient satisfies
\[
  m_r(x)=\nabla f(x)+r^2A_x+e_r(x),
  \qquad
  \norm{e_r(x)}\le\frac{r^3}{6},
\]
where
$A_x=\frac12\E_VD^3f(x)[V,V,\cdot]$.
The estimator \eqref{eq:beta4-main} equals
$(4G_{h/2}-G_h)/3$.  Hence
\[
  \norm{\E\widehat g_h(x)-\nabla f(x)}
  \le\frac13\left(4\frac{(h/2)^3}{6}+\frac{h^3}{6}\right)
  =\frac{h^3}{12}.
\]
Each iteration uses exactly the four independent scalar observations
\[
  Y(x+hU),\quad Y(x-hU),\quad
  Y(x+\tfrac h2U),\quad Y(x-\tfrac h2U).
\]
The deterministic component has norm at most
$9d(1/3+4/3)=15d$.  The noise component has second moment
\[
  \frac{d^2}{36h^2}\left(64\cdot2+1\cdot2\right)
  =\frac{65d^2}{18h^2}.
\]
Consequently
\[
  \E\norm{\widehat g_h(x)}^2
  \le225d^2+\frac{65d^2}{18h^2}
  \le229\frac{d^2}{h^2},
\]
which proves \eqref{eq:beta4-bounds-main}.  With $n$ iterations the exact raw query count is $4n$, not $n$ or $2n$.

\section{Other fixed nondegenerate parameters}\label{app:parameters}

Consider fixed constants
\[
  0<\mu<L_g<\infty,
  \qquad L_\beta>0,
  \qquad \sigma>0,
  \qquad R>0,
\]
and the analogous class with $\mu$-strong convexity, $L_g$-Lipschitz gradient, $L_\beta$ H\"older constant in \eqref{eq:high-smoothness}, minimizer in $\Ball(R)$, and Gaussian noise variance $\sigma^2$.  For every fixed choice of these constants and fixed finite $\beta>2$, the conclusions of \cref{thm:main} and \cref{cor:query} remain unchanged, with constants allowed to depend on $(\beta,\mu,L_g,L_\beta,\sigma,R)$.

For the lower bound, put
\[
  \lambda:=\frac{\mu+L_g}{2},
  \qquad
  \delta:=\frac{L_g-\mu}{2},
  \qquad
  \rho:=\min\{1,R\},
\]
and consider
\[
  f_s(x)
  :=\frac{\lambda}{2}\norm{x}^2
    -A(\rho h)^\beta F_{v_s}\!\left(\frac{x}{\rho h}\right),
  \qquad h\in(0,1].
\]
Choose the fixed amplitude $A>0$ so that
\[
  AM_2\le\delta,
  \qquad
  AH_\beta\le L_\beta,
  \qquad
  A\le\frac{\lambda}{4}.
\]
Because $\rho h\le1$, the perturbation Hessian has norm at most $AM_2$, while its $k$th derivative has H\"older constant at most $AH_\beta$.  Thus every member of the family is $\mu$-strongly convex and $L_g$-smooth and satisfies the required higher-order condition.  Its minimizer is
\[
  x_s^\star=\frac{A}{\lambda}(\rho h)^{\beta-1}v_s.
\]
The last amplitude condition places this point in the flat part of the cutoff and in $\Ball(R)$.  The neighboring-signal identity becomes
\[
  \sum_{i=1}^d
  \bigl(f_s(x)-f_{s^{(i)}}(x)\bigr)^2
  \le\frac{4A^2(\rho h)^{2\beta}}{d}.
\]
Since the KL divergence between $N(m,\sigma^2)$ and $N(m',\sigma^2)$ is $(m-m')^2/(2\sigma^2)$, choosing
\[
  h=\min\left\{1,
  \left(\frac{\sigma^2d^2}{16A^2\rho^{2\beta}T}\right)^{1/(2\beta)}
  \right\}
\]
and repeating \cref{app:lower-full} yields the same $d,T$ exponents.  The factors $A,\rho,\lambda,\mu$, and $\sigma$ enter only the fixed multiplicative constant.

For the upper bound, project onto $\Ball(R)$ and use radii $r_j\rho h$.  Every query then lies in $\Ball(R+\rho)$, where
$\norm{\nabla f}\le L_g(2R+\rho)$.  The multiscale estimator satisfies
\[
  \norm{\E\widehat g_h(x)-\nabla f(x)}
  \le C_\beta L_\beta(\rho h)^{\beta-1},
  \qquad
  \E\norm{\widehat g_h(x)}^2
  \le C\frac{d^2}{h^2},
\]
where $C$ depends only on the displayed fixed parameters and on $\beta$.  With steps $4/[\mu(t+1)]$, the general-$\mu$ version of \cref{lem:biased-sgd-main} gives
\[
  \E[f(\overline x_n)-f(x^\star)]
  \le \frac{B^2}{\mu}+\frac{4M}{\mu(n+1)}
  \le C'\left(h^{2\beta-2}+\frac{d^2}{nh^2}\right).
\]
For $T\ge2qd^2$, take $n=\lfloor T/(2q)\rfloor$ and $h=(d^2/n)^{1/(2\beta)}$, exactly as in \cref{sec:upper}. This uses at most $T$ values and gives error at most $C''(d^2/T)^{p_\beta}$. For $T<2qd^2$, return the origin: smoothness and $\norm{x^\star}\le R$ give $f(0)-f(x^\star)\le L_gR^2/2$, whereas $\min\{1,(d^2/T)^{p_\beta}\}\ge(2q)^{-p_\beta}$. Enlarging the fixed constant therefore proves the upper bound for every budget, including budgets smaller than one full iteration.

The strict inequality $\mu<L_g$ is a convenient nondegeneracy condition for the localized lower family.  If $L_g=\mu$, the globally $\mu$-strongly convex and $\mu$-smooth class collapses to quadratics with fixed Hessian, and unrestricted far-away queries can amplify linear coefficients.  Similarly, the fixed-positive-noise theorem is not uniform in the zero-noise limit.

\end{document}